\documentclass[12pt,a4paper]{amsart}

 \usepackage{amssymb, amstext, amscd, amsmath, amsfonts, amsthm, amscd, color}

\usepackage{tikz,mathdots,enumerate}

\usepackage{graphicx, array, blindtext}
\usepackage{url}
\usepackage{tikz}
\usetikzlibrary{shapes.geometric, calc}
\usepackage{caption}
\usepackage{array}
\usepackage{epsfig}
\usepackage{eucal}
\usepackage{latexsym}
\usepackage{mathrsfs}
\usepackage{textcomp}
\usepackage{verbatim}

\newtheorem{thm}{Theorem}[section]

\newtheorem{cor}[thm]{Corollary}

\newtheorem{lemma}[thm]{Lemma}

\newtheorem{prop}[thm]{Proposition}

\numberwithin{equation}{section}

\theoremstyle{definition}

\newtheorem{rem}[thm]{Remark}

\newcommand{\bC}{{\mathbb{C}}}

\newcommand{\bR}{{\mathbb{R}}}

\newcommand{\bZ}{{\mathbb{Z}}}

  \newcommand{\A}{{\mathcal{A}}}
  
  \newcommand{\C}{{\mathcal{C}}}

  \newcommand{\I}{{\mathcal{I}}}
  
  \newcommand{\K}{{\mathcal{K}}}
\renewcommand{\L}{{\mathcal{L}}}
  
  \newcommand{\N}{{\mathcal{N}}}

\renewcommand{\S}{{\mathcal{S}}}
  
  \newcommand{\U}{{\mathcal{U}}}

\newcommand{\ol}{\overline  }

\newcommand{\esssupp}{\operatorname{ess-sup}}

\usepackage{verbatim}

\usepackage{tikz}
\usetikzlibrary{calc}

\begin{document}


\title[The operator algebra for
translation, dilation and multiplication]{The operator algebra generated by the
translation, dilation and multiplication semigroups}



\author[E. Kastis  and S. C. Power]{E. Kastis and S. C. Power}
\address{Dept.\ Math.\ Stats.\\ Lancaster University\\
Lancaster LA1 4YF \\U.K. }

\email{l.kastis@lancaster.ac.uk}
\email{s.power@lancaster.ac.uk}



\begin{abstract}
The weak operator topology closed operator algebra on $L^2(\bR)$ generated by the one-parameter semigroups for translation, dilation and multiplication by  $e^{i\lambda x}, \lambda \geq 0,$ is shown to be a reflexive operator algebra, in the sense of Halmos, with invariant subspace lattice equal to a binest. 
This triple semigroup algebra,  $\A_{ph}$, is antisymmetric in the sense that $\A_{ph}\cap \A_{ph}^*=\bC I$, it  
has a nonzero proper weakly closed ideal generated by the finite-rank operators, and its unitary automorphism group is $\bR$. Furthermore, the $8$ choices of semigroup triples provide $2$ unitary equivalence classes of operator algebras, with $\A_{ph}$ and $\A_{ph}^*$ being chiral representatives.
\end{abstract}
\date{}

\maketitle

\section{Introduction}
Let $D_\mu$ and $ M_\lambda$  be the unitary operators on the Hilbert space
$L^2(\bR)$ given by
\[
D_\mu f(x) = f(x-\mu),\quad  M_\lambda f(x) = e^{i\lambda x}f(x)
\]
where  $\mu, \lambda$ 
are real. As is well-known, the $1$-parameter unitary groups
$\{D_\mu, \mu \in \bR\}$ and $\{M_\lambda, \lambda \in \bR\}$  provide an irreducible representation of the Weyl-commutation relations, $M_\lambda D_\mu
= e^{i\lambda \mu} D_\mu M_\lambda$, and the weakly closed operator algebra they generate is the von Neumann algebra  $B(L^2(\bR))$ of all bounded operators. (See Taylor \cite{tay}, for example.) On the other hand it was shown by Katavolos and Power in \cite{kat-pow-1} that the weakly closed nonselfadjoint operator algebra  generated by the semigroups for $\mu \geq 0$ and $ \lambda \geq 0$ is a proper subalgebra containing no self-adjoint operators, other than real multiples of the identity, and no nonzero finite rank operators. We consider here an intermediate weakly closed operator algebra  which is generated by the  semigroups for $\mu \geq 0$ and $ \lambda \geq 0$,
together with the semigroup of dilation operators  $V_t, t\geq 0,$ where
\[
V_t f(x) = e^{t/2}f(e^tx).
\]
Our main result is that this operator algebra is reflexive in the sense of Halmos (see \cite{rad-ros})
and, moreover, is equal to $Alg \L$, the algebra of operators that leave invariant each subspace in the  lattice $\L$  of closed subspaces given by
\[
 \L=  \{0\}\cup\{L^2(-\alpha,\infty), \alpha\geq 0\}\cup \{e^{i\beta x}H^2(\bR), \beta \geq 0\}\cup\{L^2(\bR)\}
\]
where  $H^2(\bR)$ is the usual Hardy space for the upper half plane. This lattice is a binest, being the union of two complete nests of closed subspaces.
\bigskip

We denote the triple semigroup algebra by $\A_{ph}$ since it is generated by $\A_{p}$, the operator algebra for the translation and multiplication semigroups, and $\A_{h}$, the operator algebra for the  multiplication and dilation semigroups. The notation reflects the fact that translation unitaries here are induced by the biholomorphic automorphims of the upper half plane which are of parabolic type,
and the dilation unitaries are induced by those of hyperbolic type.
The hyperbolic algebra
 $\A_{h}$ was considered by Katavolos and Power in \cite{kat-pow-2} and the invariant subspace lattice
$Lat \A_{h}$, viewed as a lattice of projections with the weak operator topology, was identified as  a $4$-dimensional manifold.
See also Levene and Power \cite{lev-pow-2} for an alternative derivation.

The operator algebras considered here are basic examples of   Lie semigroup algebras \cite{kat-pow-1} by which we mean  a weak operator topology closed algebra generated by the image of a Lie semigroup in a unitary representation of the ambient Lie group.
A complexity in the analysis of such algebras, defined in terms of generators, is the task of constructing operators within them with prescribed properties.  Establishing reflexivity can provide a route to constructing such operators and thereby deriving further algebraic properties.
The reflexivity of the hyperbolic algebra, that is, the identity 
 $\A_{h} = Alg Lat\A_{h}$, was obtained by Levene and Power in
 \cite{lev-pow-1} while the reflexivity of the parabolic algebra $\A_{p}$ was shown earlier in \cite{kat-pow-1}. 
We also note that Levene \cite{lev-2} has shown the reflexivity of the Lie semigroup operator algebra of $SL_2(\bR_+)$ for its standard representation on
$L^2(\bR)$ in terms of the composition operators of biholomorphic automorphisms.

The parabolic algebra $\A_{p}$ in fact  coincides with the Fourier binest algebra $Alg \L_{\rm FB}$, the reflexive algebra for the lattice $\L_{\rm FB}$, the Fourier binest, given by
\[
\L_{\rm FB} = \{0\}\cup \{L^2(-\alpha,\infty), \alpha\in \bR\}\cup \{e^{i\beta x}H^2(\bR), \beta \in \bR\}\cup\{L^2(\bR)\}
\]
With the weak operator topology for the orthogonal projections of these spaces, $\L_{\rm FB}$  is homeomorphic to the unit circle and forms the topological boundary of a bigger lattice $Lat Alg \L_{\rm FB}$, the so-called reflexive closure of $\L_{\rm FB}$. This  lattice is equal to the full lattice $Lat \A_{p}$ of all closed invariant subspaces of $\A_{p}$ and is homeomorphic to the unit disc. In contrast we see that the binest $\L$ for $\A_{ph}$ is reflexive as a lattice of subspaces; $\L = Lat Alg \L$.

A complexity in establishing the reflexivity of $\A_{p}$, $\A_{h}$ and $\A_{ph}$  
is the absence of an approximate identity of finite rank operators,
a key device in the theory of nest algebras (Davidson \cite{dav}, Erdos \cite{erd} and Erdos and Power \cite{erd-pow}). The same might be said of $H^\infty(\bR)$, the classical Lie semigroup  algebra with which these operator algebras bear some affinities.  As a substitute we identify the dense subspace  $\A_{ph}\cap \C_2$ of Hilbert-Schmidt integral operators. Also,  by exploiting  the Hilbert space  geometry of $\C_2$
we are able to identify various subspaces of $\C_2$ associated with the algebras $\A_{p}, \A_{h}, \A_{ph}$ and their containing nest algebras.

As in the analysis of $\A_p$ and $\A_h$ the classical Paley-Wiener (in the form $FH^2(\bR)=L^2(\bR_+)$) and the F. and M. Riesz theorem feature repeatedly in our arguments. Also, for the determination of the subspace  $\A_{ph}\cap \C_2$ we obtain a two-variable variant of the Paley-Wiener theorem which is of independent interest. See Corollary \ref{c:PWcorollary}. This asserts that if a function $k(x,y)$ in $L^2(\bR^2)$ vanishes on a proper cone $C$ with angle less than $\pi$, and its two-variable Fourier transform $F_2k$ vanishes on the (anticlockwise) rotated cone $R_{-\pi/2}C$,  then $k$ lies in the closed linear span of a pair of extremal subspaces with this property. These subspaces are rotations of the "quarter subspace" $L^2(\bR_+)\otimes H^2(\bR)$.  

We also obtain the following further properties.
The triple semigroup algebra $\A_{ph}$ is antisymmetric (or triangular \cite{kad-sin})
in the sense that $\A_{ph}\cap \A_{ph}^*=\bC I$. In contrast to
$\A_p$ and $\A_h$ the algebra $\A_{ph}$ contains non-zero finite rank operators and these generate a proper weak operator topology closed ideal. Also, $\A_{ph}$ has the rigidity property that its  unitary automorphism group is isomorphic to $\bR$ and implemented by the group of  dilation unitaries. 

We also see that, unlike the parabolic algebra, $\A_{ph}$ has \emph{chirality}
in the sense that  $\A_{ph}$ and  $\A_{ph}^*$ are not unitarily equivalent despite being  the reflexive algebras of \emph{spectrally isomorphic} binests. Furthermore
the $8$ choices of triples of continuous proper semigroups from
$\{M_\lambda, \lambda \in \bR\}$, $\{D_\mu:\mu \in \bR\}$ and
$\{V_t:t \in \bR\}$ give rise to exactly $2$ unitary equivalence classes of operator algebras.

\section{Preliminaries}

We start by introducing notation and terminology and by recalling some basic facts about the parabolic algebra, its subspace of Hilbert-Schmidt operators and its invariant subspaces.
 
The Volterra nest  $\mathcal{N}_v$ is the continuous nest consisting of the subspaces $L^2([\lambda,+\infty))$, for $\lambda\in\bR$, together with the trivial subspaces $\{0\},L^2(\bR)$. The analytic nest $\N_a$ is defined to be the unitarily equivalent nest $F^\ast \N_v$, where $F$ is Fourier-Plancherel transform with
\[
Ff(x)=\frac{1}{\sqrt{2\pi}}\int_\bR f(t)e^{-itx}dt
\] 
By the Paley-Wiener theorem the analytic nest consists of the chain of subspaces 
\[
e^{isx}H^2(\bR), \quad s \in \bR,
\] 
together with the trivial subspaces. These nests determine  the Volterra nest algebra $\A_v=Alg\N_v$ and the analytic nest algebra $\A_a=Alg\N_a$ both of which are reflexive operator algebras.

The Fourier binest is the subspace lattice
 \begin{align*}
 \L_{FB}=\N_v\cup\N_a
 \end{align*}
 and the Fourier binest algebra $\A_{FB}$ is the non-selfadjoint algebra $Alg \L_{FB}$ of operators which leave invariant each subspace of $\L_{FB}$. It is elementary to check that $\A_{FB}$ is a reflexive algebra, being the intersection of two reflexive algebras. Also, since the spaces $e^{i\beta x}H^2(\bR)$ and 
$L^2(\gamma,\infty)$ have trivial intersections 
it is elementary to see that $\A_{FB}$ contains no non-zero finite rank operators and is an antisymmetric operator algebra. 
 
The parabolic algebra $\A_p$ is defined 
as the weak operator topology closed operator algebra 
on $L^2(\bR)$ that is generated by the two strong operator topology continuous unitary semigroups $\{M_\lambda, \lambda\geq 0\}$, $\{D_\mu,\,\mu\geq 0\}$. 
Since the generators of $\A_p$ leave the subspaces of the binest
$\L_{FB}$ invariant, we have 
 $\A_p\subseteq\A_{FB}$. Katavolos and Power showed in \cite{kat-pow-1} that these two algebras are equal and we next give the proof of this from Levene \cite{lev}.
 
Write $\C_2$ for the ideal of Hilbert-Schmidt operators on $L^2(\bR)$ and let $Intk$ denote the Hilbert-Schmidt integral operator given by 
\begin{align*}
(Intk\,f)(x)=\int_\bR k(x,y)f(y)dy
\end{align*} 
where $k\in L^2(\mathbb{R}^2)$. Also let $\Theta_p$ be the unitary operation on the space of kernel functions $k(x,y)$ given by
$\Theta_p(k)(x,t)= k(x,x-t)$. Since a Hilbert-Schmidt operator in $\A_p$ lies in both
the nest algebras $Alg \N_v$ and $Alg \N_a$ and in this sense is doubly upper triangular, it is straightforward to verify the following inclusion.
 
\begin{prop}\label{p:thetamap}The subspace of Hilbert-Schmidt operators in the Fourier binest algebra satisfies the inclusion
\begin{align*}
\A_{FB}\cap \C_2\subseteq\{Intk\,|\,\Theta_p(k)\in H^2(\mathbb{R})\otimes L^2(\mathbb{R}_+)\}
\end{align*}
\end{prop}

For $h\in H^2(\bR)$ and $ \phi\in L^1(\bR)\cap L^2(\bR_+)$ let $h\otimes \phi$ denote the function $(x,y)\mapsto h(x)\phi(y)$. 
Then the integral operator
$Int k$ induced by the function $k=\Theta_p^{-1}(h\otimes\phi)$
is equal to $M_h\Delta_\phi$ where $\Delta_\phi$ is the bounded operator defined by the sesquilinear form
\begin{align*}
\langle\Delta_\phi f,g\rangle=\int_\bR\int_\bR \phi(t)D_tf(x)\overline{g(x)}dxdt, \textrm{ where }f,g\in L^2(\bR).
\end{align*}
Noting that $\Delta_\phi$ lies in the weak operator topology closed algebra generated by $\{D_t, t\geq 0\}$ it follows that the  integral operator $Int k$ actually lies in the smaller  algebra $\A_{p}$.
Since the linear span of such functions $k$ with separable variables is dense in the Hilbert space $H^2(\bR)\otimes L^2(\bR_+)$  it follows from the proposition that
\begin{align*}
\A_{FB}\cap \C_2\subseteq\{Intk\,|\,\Theta_p(k)\in H^2(\mathbb{R})\otimes L^2(\mathbb{R}_+)\}\subseteq \A_p\cap\C_2\subseteq \A_{FB}\cap \C_2
\end{align*}
and so these spaces coincide.

Choose sequences  $h_n, \phi_n, n=1,2,\dots ,$ of such functions so that the operators $M_{h_n}$ and $\Delta_{\phi_n}$ are bounded in operator norm and converge to the identity in the strong operator topology. This leads to the following proposition.

\begin{prop}\label{p:bai} $\A_{p}\cap \C_2$
contains a bounded approximate identity - that is, a sequence that is bounded in operator norm and converges in the strong operator topology to the identity.
\end{prop}

Combining this fact with the identification $\A_p\cap \C_2 = \A_{FB} \cap \C_2$ we obtain the following theorem.

\begin{thm}\label{t:binestalgebra} The parabolic algebra $\A_p$ 
is equal to the Fourier binest algebra $\A_{FB}$.
\end{thm}

We now describe $Lat \A_p$ from which it follows in particular that the binest $\N_a\cup\N_v$ is not a reflexive subspace lattice. See Fig. \ref{Latpar}.
 
Let
$K_{\lambda,s}=M_\lambda M_{\phi_s}H^2(\bR)$ where $\phi_s(x)=e^{-isx^2/2}$. This is evidently an invariant subspace for the multiplication semigroup and for $s\geq 0$ one can check that it is invariant for the translation semigroup. Thus
for $s\geq 0$ the nest 
$\N_s=M_{\phi_s}\N_a$ is contained in $Lat \A_p$ and these nests are distinct. In fact any two nontrivial subspaces from nests with distinct $s$ parameter have trivial intersection. With the strong operator topology for the associated subspace projections it can be shown that the set of these nests for $s\geq 0$, together with the Volterra nest
$\N_v$, is homeomorphic to the closed unit disc,
as indicated in Figure \ref{f:Latpar}.
A cocycle argument given in   \cite{kat-pow-1} leads to the fact that every invariant subspace for $\A_p$ is one of these subspaces. Thus we have

\begin{equation}\label{LatAp}
Lat\A_p=\{K_{\lambda,s}|\lambda\in\bR,s\geq 0\}\cup \N_v
\end{equation}

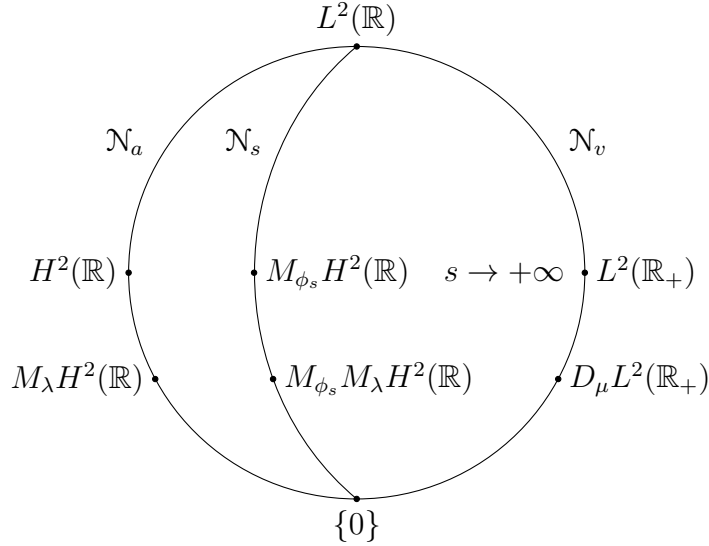
\begin{figure}[h!]
\begin{center}
\begin{tikzpicture}
\filldraw [black] (3,6) circle (1pt);
\draw (3,6) node [above]  {$L^2(\bR)$};
\filldraw [black] (3,0) circle (1pt);
\draw (3,0) node [below] {$\{0\}$};
\filldraw [black] (0,3) circle (1pt);
\draw (0,3) node [left]  {$H^2(\bR)$};
\filldraw [black] (6,3) circle (1pt);
\draw (6,3) node [right]  {$L^2(\bR_+)$};
\filldraw [black] (1.65,3) circle (1pt);
\draw (1.65,3) node [right] {$M_{\phi_s}H^2(\bR)$};
\filldraw [black] (0.35,1.59) circle (1pt);
\draw (4,3) node [right] {$s\rightarrow +\infty$};
\draw (0.35,1.59) node [left]  {$M_{\lambda}H^2(\bR)$};
\filldraw [black] (5.65,1.59) circle (1pt);
\draw (5.65,1.59) node [right]  {$D_\mu L^2(\bR_+)$};
\filldraw [black] (1.9,1.59) circle (1pt);
\draw (1.9,1.59) node [right]  {$M_{\phi_s}M_\lambda H^2(\bR)$};
\draw (0.35,4.41) node [above left] {$\N_a$};
\draw (1.9,4.41) node [above left] {$\N_s$};
\draw (5.65,4.41) node [above right] {$\N_v$};
\draw (3,3) circle [radius=3.0];
\draw (3,0) to [out=140,in=220] (3,6);
\end{tikzpicture}
\end{center}
\caption{Parametrising $Lat \A_p$ by the unit disc.}
\label{f:Latpar}
\end{figure}

\section{Antisymmetry}
We now show that $\A_{ph}$, like its subalgebras $\A_{p}$ and $\A_{h}$,
is an antisymmetric operator algebra. 
In fact we shall prove that the containing algebra  $Alg \L$ is antisymmetric. A key step of the proof is the next lemma which will also be useful in the analysis of unitary automorphisms.
We write $\bC^+$ for the set of complex numbers with positive imaginary part. 
\begin{lemma}\label{1st lem}
Let $h,g\in H^2(\bR)$, $c,d\in\bC^+$ and let
$(x+c)h(x)=(x+d)g(x)$ for almost every $x$ in a Borel set 
 $A$  of  positive Lebesgue measure. Then 
$(x+c)h(x)=(x+d)g(x)$ almost everywhere in $\bR$.
\end{lemma}
\begin{proof}
We have
\begin{align*}
(x+c)h(x)=(x+d)g(x)&\Leftrightarrow x(h(x)-g(x))+c(h(x)-g(x))+(c-d)g(x)=0\\
&\Leftrightarrow (x+c) (h(x)-g(x))+(x+c)\frac{(c-d)g(x)}{x+c}=0\\ &\Leftrightarrow 
(x+c)\left(h(x)-g(x)+\frac{(c-d)g(x)}{x+c}\right)=0.
\end{align*}
Since $\frac{1}{x+c}\in H^\infty(\bR)$ we have  $h(x)-g(x)+\frac{(c-d)g(x)}{x+c}\in H^2(\bR)$ and so it suffices to prove the following.
Given $h\in H^2(\bR)$ and $ c\in\bC^+$, with $(x+c)h(x)=0$ almost everywhere in $A$, then  $(x+c)h(x)=0$ almost everywhere.  This is evident from basic properties of functions in  $H^2(\bR)$.
\end{proof}

In the next proof we write $D_g$ for the operator $FM_gF^*$ with $g \in H^\infty(\bR)$. This lies in the weak operator topology closed algebra generated by the operators $D_\mu = FM_\mu F^*$, for $\mu \geq 0$, and so belongs to $\A_p$ and to $Alg \L$.

\begin{thm}
The selfadjoint elements of $Alg \L$ are real multiples of the identity.
\end{thm}
\begin{proof}
Let $A\in Alg \L\cap (Alg \L)^*$.  Then $A$ is reduced by subspaces $L^2(-\mu, +\infty)$, for $\mu\geq 0$, and $M_\lambda H^2(\bR)$, for $\lambda\geq 0$. It follows
that $A$ admits two direct sum decompositions
\begin{align*}
 A=P_{L^2(\bR_-)}M_fP_{L^2(\bR_-)}+P_{L^2(\bR_+)}XP_{L^2(\bR_+)}=P_{H^2(\bR)}D_gP_{H^2(\bR)}+
P_{\overline{H^2(\bR)}}YP_{\overline{H^2(\bR)}},
\end{align*}
where $f\in L^\infty(\bR_-)$ and $ g\in H^\infty(\bR)$. 
Let $h(x)=\frac{1}{x+c}$ with $c\in \bC^+$. Then, by the first decomposition,
\begin{align*}
Ah &=M_fh+P_{L^2(\bR_+)}XP_{L^2(\bR_+)}h,\\
h^{-1}Ah &=f+h^{-1}P_{L^2(\bR_+)}XP_{L^2(\bR_+)}h
\end{align*}
and so for $x$ in $\bR_-$ we have $h^{-1}(x)(Ah)(x) =f(x)$. Also $Ah$ is in $H^2(\bR)$ and so by the previous lemma, $h^{-1}Ah$ is determined by $f$ and there is a function $\phi$ independent of $c$ which extends $f$. Thus $h^{-1}Ah=\phi$ and
$Ah=\phi h.$
Since the linear span of the family $\{h:\bR\rightarrow \bC\big|h(x)=\frac{1}{x+c},\,c\in\bC^+\}$ is dense in $H^2(\bR)$, we have $A\big|_{H^2(\bR)}=M_\phi\big|_{H^2(\bR)}$. However, by the second decomposition $A\big|_{H^2(\bR)}=D_g\big|_{H^2(\bR)}$. Thus, given $h_1\in H^2(\bR)\backslash\{0\}$, we have for every $\mu\in \bR$,
\begin{align*}
M_\phi D_\mu h_1=D_gD_\mu h_1=D_\mu D_g h_1 = D_\mu M_\phi h_1.
\end{align*}
Thus 
$\phi(x)h_1(x-\mu)=\phi(x-\mu)h_1(x-\mu)$ for almost every $x\in\bR$ and so $\phi(x)=c$ almost everywhere for some $c\in\bC$. Now we have $A\big|_{H^2(\bR)}= A\big|_{L^2(\bR_-)}= cI$ and it follows that $A=cI$, as required.
\end{proof}

\section{Finite rank operators in $Alg \L$}

It follows immediately from the definition of the binest $\L$ that 
the weak operator topology closed space
$$\I = P_+B(L^2(\bR))(I-Q_+)
$$ 
is contained in  $Alg\L$,
where $P_+$ and $Q_+$ are the orthogonal projections for
$L^2(\bR_+)$ and $H^2(\bR)$. 
From this and  Lemma \ref{key1}
it follows that, in contrast to the subalgebras $\A_p$ and $\A_h$, the algebra  $\A_{ph}$ contains finite rank operators.
Also,  it is straightforward to construct a pair of nonzero  operators in $\I$ whose product is zero, and so, unlike the semigroup algebra $H^\infty(\bR)$, it follows also that the triple semigroup algebra $\A_{ph}$ is not an integral domain. 

We now show that in fact the space $\I$ contains all the finite rank operators in $Alg \L$. Let
$\N_v^-$ and $\N_a^+$ be the subnests of $\N_v$ and $\N_a$ whose union is $\L$.

\begin{prop}\label{p:finiterank}
The weak operator topology closed ideal generated by the finite rank operators in $Alg\L$ is the space $\I$.
Moreover, each operator of rank $n$ is decomposable as a sum of $n$ rank one operators in $Alg\L$.
\end{prop}

\begin{proof}
 Let
\[
Int k : f \to \sum_{j=1}^n  \langle f ,  {h_j}\rangle g_j
\]
be a nonzero finite rank operator in $Alg \L$, with
$\{h_j\}$ and  $\{g_j\}$  linearly independent functions in $L^2(\bR)$.
There is some $\lambda_0\geq 0$, such that 
$M_{\lambda_0}H^2(\bR)\cap span\{g_i\,:\,i=1,\dots,n\}=\{0\}$. 
Since $M_{\lambda_0}H^2(\bR) \in \L$ it follows that if $f\in  M_{\lambda_0}H^2(\bR)$ then 
$\langle h_i,f\rangle=0$, for every $i=1,\dots, n$. This in turn implies that 
$h_i\in M_{\lambda_0}\overline{H^2(\bR)}$.

We see now that the functions $h_i$ have full support and, moreover, their set of restrictions to $\bR_+$ is a linearly independent set of functions. Thus there are functions  $f_1,\dots , f_n$  in $L^2(\bR_+)$ with 
$\langle f_i , {h_j}\rangle = \delta_{ij}$. Since $Int k$ is in $Alg \N_v^-$ it follows that each function $g_i$ lies in $L^2(\bR _+)$.

Since $Int k \in Alg \N_a^+$ it now follows that if
$f \in H^2(\bR)$  then $\langle f ,  h_j\rangle = 0$ for each $j$. This holds for all such $f$ and so  $h_j \in {H^2(\bR)}^\bot$ for each $j$. Since $\I \subseteq Alg \L$ the rank one operators determined by the $h_j$ and $g_j$ lie in $Alg \L$  and the second assertion of the proposition follows. The first assertion follows from this. 
\end{proof}

As we will see in the next section, the ideal $\I$ plays a key role in the proof of reflexivity of the triple semigroup algebra.

\section{Reflexivity}
We now show that the algebra $\A_{ph}$ is reflexive, that is  
$\A_{ph}=AlgLat\A_{ph}$, and for this it will be sufficient to show that
$\A_{ph}$ is the binest algebra $Alg\L$.  Figure \ref{f:binestL}  depicts the inclusion of $Lat A_{ph}$ in $Lat A_p$ implied by the following  lemma.
\begin{lemma}\label{l:LatAph}
$Lat \A_{ph}= \L$
\end{lemma}
\begin{proof}
Since $\A_{ph}$ is a superalgebra of $\A_p$ we have $Lat \A_{ph}\subseteq Lat\A_{p}$. Given a subspace $K\in Lat\A_{p}$, as in Eq. \eqref{LatAp}, there are two  cases to consider.

Suppose first that $K=M_\lambda M_{\phi_s} H^2(\bR)$, where $\phi_s(x)=e^{-isx^2/2}$, where $s\geq 0, \lambda\in\bR$. Then $K\in Lat \A_{ph}$ if and only if  $V_tK \subseteq K$ for $t\geq 0$. Given $f\in H^2(\bR)$, we have
\begin{align*}
V_t(e^{-isx^2/2}e^{i\lambda x}f(x))=
e^{t/2}e^{-is(e^tx)^2/2}e^{i\lambda (e^tx)}f(e^tx)=   
e^{t/2}e^{-i(se^{2t})x^2/2}e^{i(\lambda e^t)x}f(e^tx)
\end{align*}
Thus $V_tK\subseteq K$ if and only if $s=0$ and $\lambda\geq 0$.

For the second case let $K=L^2[\alpha,+\infty)$, for $\alpha\in\bR$. Then $V_tK\subseteq K$ if and only if $\alpha\leq 0$ and so the proof is complete.
\end{proof}

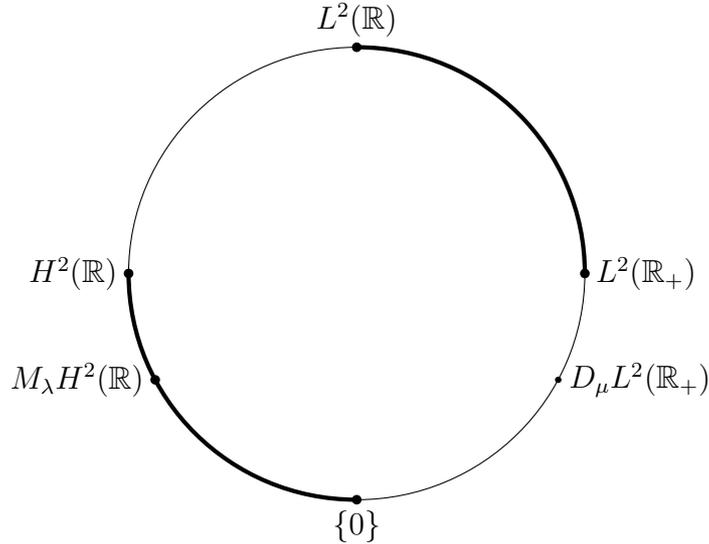
\begin{figure}[h!]
\begin{center}
\begin{tikzpicture}
\filldraw [ultra thick] (3,6) circle (1pt);
\draw (3,6) node [above]  {$L^2(\bR)$};
\filldraw [ultra thick] (3,0) circle (1pt);
\draw (3,0) node [below] {$\{0\}$};
\filldraw [ultra thick] (0,3) circle (1pt);
\draw (0,3) node [left]  {$H^2(\bR)$};
\filldraw [ultra thick] (6,3) circle (1pt);
\draw (6,3) node [right]  {$L^2(\bR_+)$};
\draw (0.35,1.59) node [left]  {$M_{\lambda}H^2(\bR)$};
\filldraw [black] (5.65,1.59) circle (1pt);
\draw (5.65,1.59) node [right]  {$D_\mu L^2(\bR_+)$};
\filldraw [ultra thick] (0.35,1.59) circle (1pt);
\draw (3,3) circle [radius=3.0];
\draw [very thick,ultra thick](3,0) to [out=180,in=270] (0,3);
\draw [very thick,ultra thick](6,3) to [out=90,in=0] (3,6);
\end{tikzpicture}
\end{center}
\caption{The binest $\L$ shown (in bold lines) as a subset of the Fourier binest.}
\label{f:binestL}
\end{figure}


Since $\A_{ph}\subseteq Alg\L$ is evident, it suffices to prove the converse inclusion. Our strategy is once again to identify the Hilbert Schmidt operators
in these two algebras. 

Given a function $k\in L^2(\bR^2)$ let $k_F,\,k_{F^\ast}$  and $V_tk$ denote the kernel functions of 
the integral operators $FIntkF^\ast,\, F^\ast Intk F$ and $V_tIntk$ respectively. We now note that $k_F=JF_2k$, where
$J$ is the flip operator, with $(Jf)(x,y)=f(x,-y)$, and $F_2$ is the two-dimensional Fourier transform
\begin{align*}
(F_2f)(\xi,\omega)=\frac{1}{2\pi}\int_{\bR}\int_{\bR}f(x,y) e^{-i(x\xi+y\omega)}dxdy 
\end{align*}
Indeed
\begin{align*}
(FIntkF^\ast)(x)&=\frac{1}{\sqrt{2\pi}}\int_\bR (Intk F^\ast f)(y) e^{-ixy} dy\\&
=\frac{1}{\sqrt{2\pi}}\int_\bR\left(\int_\bR k(y,\omega) (F^\ast f)(\omega)d\omega\right) e^{-ixy} dy\\&
=\frac{1}{2\pi}\int_\bR\left(\int_\bR k(y,\omega)\left(\int_\bR f(\xi)e^{i\omega\xi}d\xi\right)d\omega\right) e^{-ixy} dy \\&
=\frac{1}{2\pi}\int_\bR\left(\int_\bR \int_\bR k(y,\omega)e^{-ixy}e^{i\omega\xi}dy d\omega\right) f(\xi) d\xi\\&
= \int_\bR (F_2k)(x,-\xi)f(\xi)f\xi.
\end{align*}
The significance  of the above observation is that we can make use of properties of the 2D Fourier transform, and especially the fact that 
it commutes with the rotation operators. That is
\begin{align*}
F_2R_\theta=R_\theta F_2
\end{align*}
where $R_\theta$ represents the operator of clockwise rotation, for $\theta>0$, and  $\theta\in[-\pi,\pi)$. Considering the rotation operators as acting on the space of the 
kernel functions we have the following reformulation of the characterization of the Hilbert-Schmidt operators of the parabolic algebra;
\begin{align*}
\A_p\cap{\C_2}= \{Intk\,:\,k\in R_{-\pi/4}(H^2(\bR)\otimes L^2(\bR_-))\}= \{Intk\,:\,k_F\in R_{\pi/4}(L^2(\bR_+)\otimes \overline{H^2(\bR)})\}.
\end{align*} 
The convenience of the above characterization is apparent in the next lemma.
\begin{lemma}\label{key1}
Let $Intk$ lie in the ideal $\I$ generated by  the finite rank operators of $Alg\L$. Then $Intk\in \A_{ph}\cap {\C_2}$.
\end{lemma}
\begin{proof}
It follows from the previous section that $k\in L^2(\bR_+)\otimes H^2(\bR)$ and so  $k_F$ is an element of $\overline{H^2(\bR)}\otimes L^2(\bR_-)$. 
Without loss of generality we may assume that $k_F(x,y)=h(x)g(y)$, where $h\in\overline{H^2(\bR)}$, 
$g\in L^2(\bR_-)$. Define for every $t\geq0$ the functions 
\begin{align*}
h_t(x)=V_t h(x)=e^{t/2}h(e^tx) ~~~~\text{ and }~~~~g_t(y)=g(-y).
\end{align*}
Consequently, each function $k_F^t(x,y)=h_t(x)g_t(x-y)$ lies in $R_{-\pi/4} (\overline{H^2(\bR)}\otimes L^2(\bR_-))$. Since this space can be written as $R_{\pi/4}(L^2(\bR_+)\otimes \overline{H^2(\bR)})$, it follows that $Int k^t\in A_p\cap\C_2$ where $k^t = (k_F^t)_{F^*}$.
Therefore, since $V_tIntk=F^\ast V_{-t}Intk_F F$, it suffices to show that $V_{-t}k_F^t$ converges in norm to $k_F$. 
\begin{align*}
V_{-t}k_F^t(x,y)&=e^{-t/2}\,k_F^t(e^{-t}x,y)=e^{-t/2}\,h_t(e^{-t}x)\,g_t(e^{-t}x-y)\\&
=e^{-t/2}e^{t/2}h(e^te^{-t}x)g(y-e^{-t}x)e^{-t/2}=h(x)g(y-e^{-t}x)\rightarrow h(x)g(y),
\end{align*}
as $t\rightarrow +\infty$. By the dominated convergence theorem, $V_{-t}Intk_F^t$ converge to $Intk_F$ and hence $Intk\in \A_{ph}\cap\C_2$.
\end{proof}

The next lemma is crucial for the proof of the reflexivity of the triple semigroup algebra and also yields the two-variable variant of the Paley-Wiener theorem given in Corollary \ref{c:PWcorollary}.

Given $\theta_0\in[0,\pi)$, let 
\begin{align*}
Q_1^{\theta_0}&=\left\{(x,y)\in\bR^2: \arctan(y/x)\in\left[-\frac{\pi}{2}-\theta_0,\frac{\pi}{2}\right]\right\}\\
Q_2^{\theta_0}&=\left\{(x,y)\in\bR^2: \arctan(y/x)\in [-\pi,\theta_0]\right\}.
\end{align*}
Define also the set
\begin{align*}
\K_{\theta_0}=\{k\in L^2(\bR^2):\esssupp \,k\subseteq Q_1^{\theta_0}\}\cap
\{k\in L^2(\bR^2):\esssupp\,k_F \subseteq Q_2^{\theta_0}\}
\end{align*}
(see Figure \ref{f:support}) and the set
\begin{align*}
\S_{\theta_0}=\overline{span\{R_\theta( L^2(\bR_+)\otimes H^2(\bR)),\,\theta\in\{0,\theta_0\}\}}^{\|\cdot\|}.
\end{align*}

\begin{figure}[h]
\begin{center}
\begin{tikzpicture}
\path [fill=lightgray] (0,0) to (3.5,0) -- (3.5,0) to (3.5,3.5) -- (3.5,3.5) to (1.75,3.5) -- (1.75,3.5) to (1.75,1.75) -- (1.75,1.75) to (0,3.5) -- 
(0,3.5) to (0,0);
\draw [<->] (0,1.75) -- (3.5,1.75);
\draw (3.3,1.75) node [below] {$x$};
\draw [<->] (1.75,0) -- (1.75,3.5);
\draw (1.75, 3.3) node [left] {$y$};
\draw (2,1) node [right] {$\esssupp\,k$};
\draw [<-] (1.45,2.05) to [out=250, in=180] (1.75,1.25);
\draw (0.75,1) node [right] {$\theta_0$};

\draw (5.25,1.75) node [right] {$\bigcap$};

\path [fill=lightgray] (7.5,0) to (11,0) -- (11,0) to (11,3.5)-- (11,3.5) to (7.5,3.5) -- (7.5,3.5) to (9.25,1.75) -- (9.25,1.75) to (7.5,1.75) -- 
(7.5,1.75) to (7.5,0);
\draw [<->] (7.5,1.75)--(11,1.75);
\draw [<->] (9.25,0) -- (9.25,3.5);
\draw (10.8,1.75) node [below] {$x$};
\draw (9.25, 3.3) node [left] {$y$};
\draw [<-] (8.95,2.05) to [out=60, in=90] (9.75,1.75);
\draw (9.85,2) node [right] {$\theta_0$};
\draw (9.75,1) node [right] {$\esssupp\,k_F$};

\end{tikzpicture}
\end{center}
\caption{A function $k\in L^2(\bR^2)$ is an element of $\K_{\theta_0}$, if and only if both $\esssupp \,k$ and $\esssupp \,k_F$  lie in the respective shaded areas.}
\label{f:support}
\end{figure}
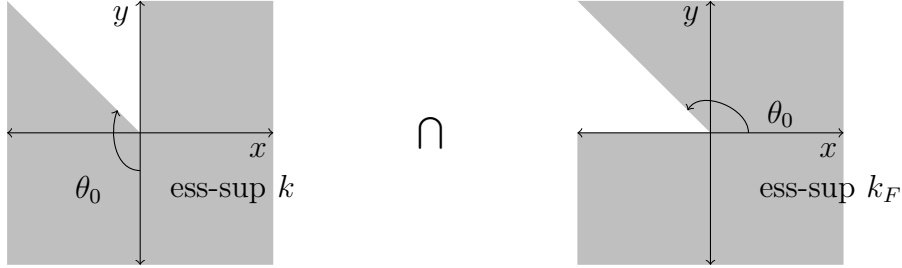

\begin{lemma}\label{key2}
$\K_{\theta_0}\,=\,\S_{\theta_0}$, for every $\theta_0\in[0,\pi)$.
\end{lemma}
\begin{proof}
Let $k\in R_\theta( L^2(\bR_+)\otimes H^2(\bR))$, with $\theta\in\{0,\theta_0\}$. Then
$\esssupp f \subseteq Q_1^{\theta_0}$. Also the function $k_F$ lies in the space $JF_2 R_\theta( L^2(\bR_+)\otimes H^2(\bR))$,
which can be written as
\begin{align*}
JF_2 R_\theta( L^2(\bR_+)\otimes H^2(\bR))=R_{-\theta}JF_2( L^2(\bR_+)\otimes H^2(\bR))=
R_{-\theta}( \overline{H^2(\bR)}\otimes L^2(\bR_-)).
\end{align*}
Hence $\esssupp\,k_F\subseteq  Q_2^{\theta_0}$, and so it follows that $\S_{\theta_0}\subseteq \K_{\theta_0}$.

To prove the other inclusion, assume that there is a function $k\in \K_{\theta_0}\,\cap\,\S_{\theta_0}^\perp$. Then the Hilbert space geometry of  $L^2(\bR^2)$ ensures that 
\begin{align}\label{eqHG}
\|k+ k_S\|>\|k\|, \,\forall\, k_S\in \S_{\theta_0}\backslash \{0\}.
\end{align}
Define now the orthogonal projections $P_\theta=proj(R_{-\theta}(L^2(\bR)\otimes L^2(\bR_-)),\,\theta\in \{0,\theta_0,\pi/2\}$. 
Noting that $P_{\pi/2}=proj(L^2(\bR_+)\otimes L^2(\bR))$, decompose $k$ as the sum of two orthogonal parts,
\begin{align*}
k=P_{\pi/2}\,k+P^\bot_{\pi/2}\,k
\end{align*} 
where $P^\bot_{\pi/2}=I-P_{\pi/2}$. Applying to both sides the operator $JF_2$, we have
\begin{align*}
k_F=(P_{\pi/2}\,k)_F+(P^\bot_{\pi/2}\,k)_F.
\end{align*}

Also
\begin{align}\label{eqNormSq}
\|k_F\|=\|P_0(P_{\pi/2}\,k)_F\|^2+\|P_0^\bot(P_{\pi/2}\,k)_F\|^2
+\|P_{\theta_0}(P^\bot_{\pi/2}\,k)_F\|^2 +
\|P_{\theta_0}^\bot(P^\bot_{\pi/2}\,k)_F\|^2
\end{align}
Since $P_0(P_{\pi/2}\,k)_F\in \, \overline{H^2(\bR)}\otimes L^2(\bR_-)$
which is the space $JF_2(L^2(\bR_+\otimes H^2(\bR))$, it follows that $(P_0(P_{\pi/2}\,k)_F)_{F^*}$ lies in $S_{\theta_0}$.
Since subtraction of this function from $k$ cannot decrease the norm of $k$, in view of Eq.\eqref{eqHG} and Eq.\eqref{eqNormSq} , it follows that this function is the zero function.

Similarly, taking into account that $k\in L^2(Q_1^{\theta_0})$, we have $P^\bot_{\pi/2}\,k\in R_{\theta_0}(L^2(\bR_+)\otimes L^2(\bR))$, which implies that 
$(P^\bot_{\pi/2}\,k)_F$ lies in $R_{-\theta_0}(\overline{H^2(\bR)}\otimes L^2(\bR))$. Therefore,
\begin{align*}
P_{\theta_0}(P^\bot_{\pi/2}\,k)_F\in \,R_{-\theta_0}(\overline{H^2(\bR)}\otimes L^2(\bR_-))
\end{align*}
and so $(P_{\theta_0}(P^\bot_{\pi/2}\,k)_F)_{F^*}$ lies in $S_{\theta_0}$. Subtraction of this function from $k$ cannot decrease $\|k\|$, and this subtraction corresponds to the subtraction of $P_{\theta_0}(P^\bot_{\pi/2}\,k)_F$ from $k_F$, and so  $P_{\theta_0}(P^\bot_{\pi/2}\,k)_F=0$.

We now see that
\begin{align*}
k_F=(P_{\pi/2}\,k)_F+(P^\bot_{\pi/2}\,k)_F=P_0^\bot(P_{\pi/2}\,k)_F+
P_{\theta_0}^\perp(P^\perp_{\pi/2}\,k)_F.
\end{align*}
The first function in the sum representation is in $\overline{H^2(\bR)}\otimes L^2(\bR_+)$, and is supported in the upper half plane, while the second function is supported in the half plane $y \leq -x$. However, we also have 
$k_F\in L^2(Q_2)$. These three facts are indicated in Figure \ref{supptildek} which depicts  the essential support of $k_F$ and the two forms of the semi-infinite lines on which (almost every) restriction of $k_F$ agrees with the restriction of a function in $\overline{H^2(\bR)}$.

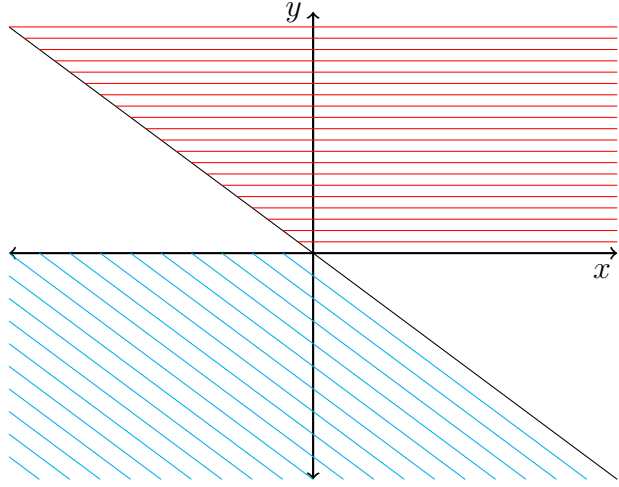
\begin{figure}[h]
\begin{center}
\begin{tikzpicture}
\draw [<->][thick] (1,3) -- (9,3);
\draw (8.8,3) node [below] {$x$};
\draw [<->][thick] (5,0) -- (5,6.2);
\draw (5, 6.2) node [left] {$y$};
\draw (1,6) -- (9,0);
\draw [red](1,6) -- (9,6);
\draw [red](1.2,5.85) -- (9,5.85);
\draw [red](1.4,5.7) -- (9,5.7);
\draw [red](1.6,5.55) -- (9,5.55);
\draw [red](1.8,5.4) -- (9,5.4);
\draw [red](2,5.25) -- (9,5.25);
\draw [red](2.2,5.1) -- (9,5.1);
\draw [red](2.4,4.95) -- (9,4.95);
\draw [red](2.6,4.8) -- (9,4.8);
\draw [red](2.8,4.65) -- (9,4.65);
\draw [red](3,4.5) -- (9,4.5);
\draw [red](3.2,4.35) -- (9,4.35);
\draw [red](3.4,4.2) -- (9,4.2);
\draw [red](3.6,4.05) -- (9,4.05);
\draw [red](3.8,3.9) -- (9,3.9);
\draw [red](4,3.75) -- (9,3.75);
\draw [red](4.2,3.6) -- (9,3.6);
\draw [red](4.4,3.45) -- (9,3.45);
\draw [red](4.6,3.3) -- (9,3.3);
\draw [red](4.8,3.15) -- (9,3.15);
\draw [cyan](4.6,3) -- (8.6,0);
\draw [cyan](4.2,3) -- (8.2,0);
\draw [cyan](3.8,3) -- (7.8,0);
\draw [cyan](3.4,3) -- (7.4,0);
\draw [cyan](3,3) -- (7.0,0);
\draw [cyan](2.6,3) -- (6.6,0);
\draw [cyan](2.2,3) -- (6.2,0);
\draw [cyan](1.8,3) -- (5.8,0);
\draw [cyan](1.4,3) -- (5.4,0);
\draw [cyan](1,3) -- (5.0,0);
\draw [cyan](1,2.7) -- (4.6,0);
\draw [cyan](1,2.4) -- (4.2,0);
\draw [cyan](1,2.1) -- (3.8,0);
\draw [cyan](1,1.8) -- (3.4,0);
\draw [cyan](1,1.5) -- (3,0);
\draw [cyan](1,1.2) -- (2.6,0);
\draw [cyan](1,0.9) -- (2.2,0);
\draw [cyan](1,0.6) -- (1.8,0);
\draw [cyan](1,0.3) -- (1.4,0);
\end{tikzpicture}
\end{center}
\caption{The essential support of $k_F$.}
\label{supptildek}
\end{figure}

Since $k_F\in L^2(Q_2)$ it follows that
\[
\|k_F\|^2 =\|P_{\theta_0}(P_{\pi/2}\,k)_F\|^2 + \|P_0(P^\perp_{\pi/2}\,k)_F\|^2
\]
and so we deduce that  $P_{\theta_0}(P_{\pi/2}\,k)_F = (P_{\pi/2}\,k)_F$
and $P_0(P^\perp_{\pi/2}\,k)_F=(P^\perp_{\pi/2}\,k)_F$.
However, 
\[
(P_{\pi/2}\,k)_F \in \overline{H^2(\bR)}\otimes L^2(\bR)\quad  
\textrm{and}\quad  (P_{\pi/2}^\perp\,k)_F\in R_{-\theta_0}(\overline{H^2(\bR)}\otimes L^2(\bR))
\]
and so both functions are equal to zero, as every 
$\ol{H^2(\bR)}$-slice is zero on a non-null interval. Consequently, $k_F=0$, which implies that $k=0$ and this fact completes the proof.
\end{proof}

\begin{cor}\label{c:PWcorollary}
Let $0<\alpha <\pi/2$ and let $C_\alpha$ be the proper cone 
of points $(x,y)$ with $ x\geq 0$ and  $|\arctan{y/x}| < \alpha$.
Then the following conditions are equivalent for a function $k\in L^2(\bR^2)$.

(i) $k$ vanishes on $C_\alpha$ and $\hat{k}$ vanishes on $R_{-\pi /2}C_\alpha$.

(ii) $k$ lies in the closed linear span of $R_{\alpha/2 }(H^2(\bR) \otimes L^2(\bR_-))$ and $R_{-\alpha/2 }(H^2(\bR) \otimes L^2(\bR_+))$.
\end{cor}

Our next goal is to make use the previous  lemma to show that 
\begin{align*}
Alg\L\cap{\C_2}\,=\, \overline{(\A_p\cap{\C_2})+\I}^{\|\cdot\|}.
\end{align*}
First, we determine the Hilbert - Schmidt operators of $Alg\L$.
\begin{lemma}\label{l:supportinclusion} 
 $Alg\L\cap{\C_2}\subseteq \K_{\pi/4}$
 \end{lemma}
 \begin{proof}
Suppose first that $k\in L^2(\bR^2)$ is a kernel function such that $Intk L^2[\lambda,+\infty)$ is a subspace of $L^2[\lambda,+\infty)$, for every $\lambda\leq0$. 
Let $x<\lambda <0,$ and
take $f\in L^2(\lambda,+\infty)$. Then
 \begin{align*}
\int_{\bR}k(x,y)f(y)dy=
(Int k f)(x) =0.
 \end{align*} 
Thus $k(x,y)=0$ for almost for every $y>\lambda$ and $\esssupp\,k\subseteq Q_1^{\pi/4}$.

Suppose next that  $k\in L^2(\bR^2)$ and  $Intk M_\lambda H^2(\bR)\subseteq M_\lambda H^2(\bR)$ for every $\lambda\geq0$. Then 
the following equivalent inclusions hold for all $\lambda > 0.$
\begin{align*}
&Intk M_\lambda H^2(\bR)\subseteq M_\lambda H^2(\bR),\\ 
&FIntk F^\ast F M_\lambda H^2(\bR)\subseteq F M_\lambda H^2(\bR),\,\\  
&FIntk F^\ast D_\lambda L^2(\bR_+)\subseteq D_\lambda L^2(\bR_+),\,\\ 
&FIntk F^\ast L^2[\lambda,+\infty)\subseteq  L^2[\lambda,+\infty).
\end{align*}
Thus $Int k_FL^2[\lambda,+\infty)\subseteq L^2[\lambda,+\infty)$, for every $\lambda\geq0$. Given $x<0$ and  $f\in L^2(\bR_+)$ we have
 \begin{align*}
\int_{\bR}k_F(x,y)f(y)dy=
(Int k_Ff)(x) =0
 \end{align*}
and so it follows that $k_F(x,y)=0$ for almost for every $y>0$.
Also, for $x\geq 0$ and $f\in L^2[\lambda,+\infty)$ with $\lambda>x$, we again have
$(Int k_Ff)(x)=0$ and so $\esssupp\,k_F\subseteq Q_2^{\pi/4}$, as required.
 \end{proof}
\begin{lemma}
The algebras $\A_{ph}\cap{\C_2}$ and $Alg\L\cap{\C_2}$ coincide.
\end{lemma}
\begin{proof}
By the previous lemma and Lemma \ref{key2}, we have $Alg\L\cap{\C_2}\subseteq \S_{\pi/4}$, where 
\begin{align*}
\S_{\pi/4}&=\overline{R_{\pi/4}( L^2(\bR_+)\otimes H^2(\bR))+ L^2(\bR_+)\otimes H^2(\bR)}^{\|\cdot\|}\\&=
\overline{(\A_p\cap{\C_2})+\I}^{\|\cdot\|}.
\end{align*}
Applying Lemma \ref{key1}, the desired inclusion follows. 
\end{proof}

We have noted in Section 2 that $\A_p\cap \C_2$ contains an operator norm bounded sequence which is an approximate identity for the space of all Hilbert-Schmidt operators. Since this sequence also lies in $\A_{ph}$ it follows from the previous lemma that the weak operator topology closures of $\A_{ph}\cap \C_2$ and $Alg \L\cap \C_2$ coincide.
Thus, the following theorem is proved.
\begin{thm}\label{t:AphReflexive}
The operator algebra $\A_{ph}$ is reflexive, with $\A_{ph} = Alg \L=\overline{\A_p+\I}^{WOT}$.
\end{thm}

\section{The unitary automorphism group of $\A_{ph}$}

In the case of the parabolic algebra the group of unitary automorphisms, $X \to AdU(X) = UXU^*$, was identified in \cite{kat-pow-1} as the $3$-dimensional  Lie group of automorphisms $Ad(M_\lambda D_\mu V_t)$ for $\lambda, \mu$ and $t$ in $\bR$.
The following theorem shows  that the larger algebra $\A_{ph}$ is similarly rigid.

\begin{thm}\label{t:unitaryauto}
The unitary automorphism group of $A_{ph}$ is isomorphic to $\bR$ and equal to $\{Ad(V_t):t\in\bR\}$.
\end{thm}
\begin{proof}
Let $Ad(U)$ be a unitary automorphism of $A_{ph}$. Since $\A_{ph} = Alg \L$   it follows from Lemma \ref{l:LatAph} that  
\begin{equation}\label{1st eq}
UH^2(\bR)=H^2(\bR),\quad  UM_\lambda H^2(\bR)=M_\mu H^2(\bR)
\end{equation}
where $\mu\geq 0$ depends on $\lambda\geq 0$ and $\mu:\bR_+\to \bR_+$ is a continuous bijection.
Also
\begin{equation}\label{2nd eq}
UL^2(\bR_-)=L^2(\bR_-),\quad UL^2(-\lambda',0)=L^2(-\mu',0)
\end{equation}
with $\mu':\bR_+\to \bR_+$ is a continuous bijection.

Note that the subspaces $L^2(-\lambda,\infty)$ of $L^2(\bR_-)$ form a continuous nest of  multiplicity one and so it follows from ~\eqref{2nd eq} and elementary nest algebra theory (see Davidson \cite{dav} for example) that the unitary operator $U$ has the form $U=M_\psi C_f \oplus U_1$, where $\psi$ is a unimodular function
in $L^\infty(\bR_-)$, 
$f:\bR_-\rightarrow \bR_-$ is a strictly increasing bijection, and $C_f$ is the unitary composition operator on $L^2(\bR_-)$ with
\begin{align*}
(C_f g)(x)=(f'(x))^{1/2}g(f(x)).
\end{align*}

Let $h\in H^2(\bR)$. Then for 
$x\in\bR_-$ we have
\begin{align*}
(UM_\lambda h)(x)=(\psi C_f M_\lambda h)(x)=
\psi(x)e^{i\lambda f(x)}(f'(x))^{1/2}h(f(x))=e^{i\lambda f(x)}(Uh)(x)
\end{align*}
Take $c\in\bC^+$ and let $h_c\in H^2(\bR)$ be the function for which $(Uh_c)(x)=\frac{1}{x+c}$. Then
\[
(UM_\lambda h_c)(x)=e^{i\lambda f(x)}\frac{1}{x+c}
\]
and so
\[
 (x+c) g_{\lambda, c}(x)=e^{i\lambda f(x)},
\]

\noindent  where $g_{\lambda, c}= UM_\lambda h_c\in H^2(\bR)$. By  Lemma \ref{1st lem} the functions  $(x+c)g_{\lambda, c}(x)$ are independent of $c$ and agree for all real $x$.
Thus there is a unique extension  of $e^{i\lambda f(x)}$
to $\bR$, say $\phi_\lambda (x)$, such that
\begin{align*}
\phi_\lambda(x)=e^{i\lambda f(x)},\,\textrm{ for almost every } x\in\bR_-
\end{align*}
and
\begin{align*}
\phi_\lambda(x)=(x+c)g_{\lambda,c}(x), \,\textrm{ for almost every } x\in\bR.
\end{align*}
It now follows  that
\begin{align*}
UM_\lambda h_c= M_{\phi_\lambda} Uh_c.
\end{align*}
Since the closed linear span of the functions $h_c=U^\ast\frac{1}{x+c},\,c\in\bC^+$, is equal to $H^2(\bR)$, we obtain
\begin{align*}
UM_\lambda H^2(\bR)= M_{\phi_\lambda} UH^2(\bR).
\end{align*} 
Now \eqref{1st eq} implies  that
\begin{align*}
M_\mu H^2(\bR)=M_{\phi_\lambda}H^2(\bR).
\end{align*}
Therefore, $\phi_\lambda$ is inner and $\phi_\lambda(x)/e^{i\mu x}$ is equal to a unimodular constant $c_\lambda=e^{i\alpha_\lambda}$ depending on $\lambda$. Thus, for every $x\in\bR_-$, we have
\begin{align*}
i\lambda f(x)-i\mu x=i\alpha_\lambda
\end{align*}
or equivalently
\begin{align*}
f(x)=\frac{\mu}{\lambda} x+\frac{\alpha_\lambda}{\lambda}.
\end{align*} 
It follows that $\alpha_\lambda=0$, since $f(0)=0$, and that $\mu = \beta \lambda$ for some positive constant $\beta$.
Thus, for $x<0$,
\begin{align*}
 (C_fh)(x)=\beta^{1/2} h(\beta x)=
(V_{\log {\beta}}h)(x).
\end{align*}
Writing $t=\log\beta$, we have $Uh=\psi V_t h+U_1h$, and so with $h(x)=\frac{1}{x+d}$ and $x<0$ we have $(Uh)(x)=\psi(x)(V_t h)(x)$
and
\begin{align*}
\frac{e^tx+d}{e^{t/2}}(Uh)(x)
&=\psi(x).
\end{align*}
By Lemma \ref{1st lem} again, $\frac{e^tx+d}{e^{t/2}}Uh$ is determined by $\psi$ and there is 
analytic function $\phi$ such that
\begin{align*}
\frac{e^tx+d}{e^{t/2}}Uh=\phi.
\end{align*}  We conclude that $Uh=\phi V_t h$ for all such $h$ and so
$\phi$ is unimodular. Since $UH^2(\bR)=H^2(\bR)$ it follows that almost everywhere $\phi$ is a unimodular constant, $\eta$ say. Thus $U=\eta V_t$ and the proof is complete.
\end{proof}

\begin{rem}\label{r:Lst}
Note that the binest $ \L_{\alpha,\beta}$ given by
\[
 \L_{\alpha,\beta}=  \{0\}\cup\{L^2(\alpha',\infty), \alpha'\leq \alpha\}\cup \{e^{i\beta' x}H^2(\bR), \beta' \geq \beta\}\cup\{L^2(\bR)\}
\]
is equal to $D_{\alpha}M_{\beta}\L$. Thus  $ \L_{\alpha,\beta}$ is unitarily equivalent to $\L$. Also the unitary operator 
$U=D_{\alpha}M_{\beta}$ provides a unitary isomorphism
$Ad U: Alg \L \to Alg \L_{\alpha,\beta}$ between their reflexive algebras.
\end{rem}

\section{Further binests}

Once again, write $\N_v^-$ and $\N_a^+$ for the subnests of $\N_v$ and $\N_a$ whose union is $\L$. Also let $\N_v^+, \N_a^-$ be the
analogous subnests of  $\N_v$ and $\N_a$ for which
$P_-=(I-P_+)$ is the atomic interval projection for $\N_v^+$
and $Q_+$ is the atomic interval projection for  $\N_a^-$.

 By the F. and M. Riesz theorem the orbit of $H^2(\bR)$ under the Fourier-Plancherel transform $F$ is the subspace $H^2(\bR)$ together with the three subspaces
\[
FH^2(\bR)= L^2(\bR_+),\ \  F^2H^2(\bR)= \ol{H^2(\bR)},\ \
F^3H^2(\bR)= L^2(\bR_-).
\]
More generally,
the lattice $Lat \A_p$, with the weak operator topology for subspace projections, forms one quarter  of the Fourier-Plancherel sphere, and  the Fourier-Plancherel transform $F$ effects a period $4$ rotation of this sphere. (See \cite{kat-pow-2,lev-pow-2}.)

We now note that there are $8$ binest lattices which are pairwise order isomorphic as lattices and which have a similar status to 
$\L = \N_a^+ \cup  \N_v^-$. These fall naturally into two groupings of $4$.
Write $J$ for the unitary operator $F^2$, so that $Jf(x)=f(-x)$. (There will be no conflict here with notation from the previous section.)
Writing $\ol{K}$ for $\{\ol{f}:f\in K\}$,
these groupings are
\[
\N_a^+ \cup  \N_v^-, \ \ \N_v^+ \cup \overline{\N_a^-}, \ \ 
\overline{\N_a^+} \cup J\N_v^-, \ \ J\N_v^+\cup \N_a^-
\]
and 
\[ 
\N_a^- \cup  \N_v^+, \ \ \N_v^- \cup \overline{\N_a^+},\ \   
\overline{\N_a^-}\cup J\N_v^+, \ \ J\N_v^-\cup \N_a^+
\] 
forming the  orbits of the subspace lattices $\N_a^+ \cup  \N_v^-$ and 
$\N_a^- \cup  \N_v^+$ under $F$.
Note that the symbols ``$+$'' and ``$-$'' indicate the 
 ``upper'' and ``lower'' choices for the atomic interval of the nest. Since  $F$ induces an order isomorphism of the lattices, $F$ respects these symbols. By Theorem \ref{t:binestalgebra} and the identities
\[
FM_\lambda F^*=D_\lambda, \quad FD_\mu F^*=M_{-\mu},\quad FV_tF^*=V_{-t}
\]
it follows that the binest algebras for these $8$ binests are (respectively) equal to weak operator closed operator algebras for the following generating semigroup choices for $\{M_\lambda\},
\{D_\mu\}$ and $\{V_t\}$: 
\[
+++\quad -+-\quad --+\quad +--
\]
\[
++-\quad -++\quad ---\quad +-+
\]

View the lattice $\L=\N_a^+ \cup  \N_v^-$ as the right-handed choice in Figure \ref{f:binestL}, write $\L_r$ for $\L$, and view
$\L_l = \N_a^- \cup  \N_v^+$ as the left-handed choice. From the observations above the $8$ binests determine either $1$ or $2$ unitary equivalence classes of triple semigroup algebras. In fact there are two classes.

\begin{thm}\label{t:notequivalent}
The triple semigroup algebra $\A_{ph}=Alg\L_r$ is not unitarily equivalent to triple semigroup algebra $\A_{ph}^* = Alg \L_l$
\end{thm}

\begin{proof}
By Theorem \ref{t:binestalgebra}, $\A_{ph}^*=(Alg (\N_a^+ \cup  \N_v^-))^*$ which is the binest algebra for the union of the nests $(\N_a^+)^\bot$ and  $  (\N_v^-)^\bot$. We have 
\[
(\N_a^+)^\bot = J\N_a^-,\quad  (\N_v^-)^\bot = J\N_v^+
\]
and so it suffices to show that the binests
\[
\N_a^+ \cup  \N_v^-,\quad \N_a^- \cup  \N_v^+
\]
are not unitarily equivalent.

Suppose, by way of contradiction, that for some unitary $U$ the binest $U(\N_a^+ \cup  \N_v^-)$ coincides with $\N_a^- \cup  \N_v^+$. Then 
\[
FU(\N_a^+ \cup  \N_v^-) = F(\N_a^- \cup  \N_v^+) = \N_v^-\cup 
\ol{\N_a^+}
\]
We have $\N_v^- = \{L^2(\lambda,\infty), \lambda \leq 0\}$ and so by elementary nest algebra theory, as in the proof of Theorem \ref{t:unitaryauto},
\[
FU = M_\psi C_f\oplus X\]
for some unimodular function $\psi$ on $\bR_-$ and a composition operator $C_f$ on $L^2(\bR_-)$ associated with a continuous bijection
$f$.

We have
\[
FU: e^{i\lambda x}H^2(\bR)  \to  e^{-i\mu x}\ol{H^2(\bR) }
\]
with $\mu = \mu(\lambda):\bR_+\to \bR_+$ a bijection.

Take $h_c\in H^2(\bR) $ such that $FUh_c = \frac{1}{x-c} \in \ol{H^2(\bR) }$, with $c \in \bC^+$.  Then, for $x<0,\lambda >0$,
\begin{align*} 
(FUM_\lambda h_c)(x)& = (M_\psi C_fM_\lambda h_c)(x),\\
(FUM_\lambda h_c)(x)& = (e^{i\lambda f(x)}M_\psi C_f h_c)(x),\\
(FUM_\lambda h_c)(x)& = e^{i\lambda f(x)}(FUh_c)(x),\\
g_{\lambda,c}(x)& = e^{i\lambda f(x)}\frac{1}{x-c},
\end{align*}
where $g_{\lambda,c}= FUM_\lambda h_c \in \ol{H^2(\bR) }$. We may apply Lemma \ref{1st lem} as in the proof of Theorem \ref{t:unitaryauto} (although to $\ol{H^2(\bR)}$ functions here) to deduce that $FU = M_\phi V_t$ for some unimodular function $\phi$ and some real $t$. Thus we obtain 
\[
\ol{H^2(\bR)}= FU(H^2(\bR))=\phi(H^2(\bR)).
\]
This implies that $\ol{H^2(\bR)}$ is invariant for multiplication
by $M_\lambda$ for $\lambda>0$, as well for $\lambda < 0$, and so is 
a reducing subspace for the full multiplication group $\{M_\lambda:\lambda \in \bR\}$. This is a contradiction, as desired, since such spaces have the form $L^2(E)$.
\end{proof}

The fact that $\A_{ph}= Alg\L_r$ and $\A_{ph}^*=Alg\L_l$ fail to be unitarily equivalent expresses the following \emph{chirality} property. We say that a reflexive operator algebra $\A$ is \emph{chiral} if 

\bigskip

(i) $\A$ and $\A^*$ are not unitarily equivalent, and 
\\

(ii)  $Lat\A$ and $Lat \A^*$ are \emph{spectrally equivalent} in the sense that there is an order isomorphism
$\theta:Lat \A \to Lat \A^*$ such that for each pair of interval projections $\{P_1-P_2, Q_1-Q_2\}$ for $Lat \A$ the projection pairs
\[
\{P_1-P_2, Q_1-Q_2\},\quad \{\theta (P_1)-\theta (P_2), \theta (Q_1)-\theta (Q_2)\}
\]
are unitarily equivalent. 
\bigskip

While the spectral invariants for a pair of projections are well-known (Halmos \cite{hal}) there is presently no analogous classification of binests.



\begin{rem}\label{r:problems}
We remark that the companion binest algebra $Alg(\N_a^+\cup \N_v^+)$, in which both nests have an upper choice for the atomic interval, has no finite rank operators and it is unclear to us whether this reflexive algebra is a proper superalgebra of $\A_p$. 

The examination  of reflexivity for nonselfadjont operator algebras has its origins in Sarason's consideration \cite{sar} of  the Banach algebra $H^\infty(\bR)$ with the weak star topology. This algebra is isomorphic to both the basic Lie semigroup algebra, for $\bR_+$, and the discrete semigroup left regular representation algebra for $\bZ_+$. In the case of noncommutative discrete  groups the property of reflexivity has been obtained in many settings, including  free semigroups
(Davidson and Pitts \cite{dav-pit-1}), free semigroupoids
 (Kribs and Power \cite{kri-pow-1}),
and the discrete Heisenberg group (Anoussis, Katavolos and Todorov \cite{ano-kat-tod}). These operator algebras satisfy double commutant theorems and partly for this reason their algebraic and spatial properties, such as semisimplicity and invariant subspace structure, are somewhat more evident than in the case for Lie semigroup algebras. We note, for example, that the following questions seem to be open.
\medskip

Question 1. (See \cite{pow-seville}.) Does $\A_p$ contain nonzero operators with product zero? 
\medskip

Question 2. Does the Jacobson radical of $\A_p, \A_h$ or $\A_{ph}$
admit an explicit characterisation bearing some analogy to Ringrose's characterisation \cite{rin} for a nest algebra ?
\medskip

Question 3. (See \cite{lev-2}.) Is the Lie semigroup algebra of
an arbitrary irreducible representation of 
$SL_2(\bR_+)$ a reflexive operator algebra?

\end{rem}

\end{document}